\documentclass[12pt,oneside,english]{amsart}
\usepackage[T1]{fontenc}
\usepackage[latin1]{inputenc}
\usepackage{geometry}
\usepackage{setspace}
\usepackage{amssymb}

\makeatletter

 \theoremstyle{plain}
\newtheorem{thm}{Theorem}[section]
  \theoremstyle{plain}
  \newtheorem{lem}[thm]{Lemma}
  \theoremstyle{remark}
  \newtheorem*{acknowledgement*}{Acknowledgement}

\numberwithin{equation}{section}

\usepackage{babel}
\makeatother
\begin{document}

\title{law of large numbers for monotone convolution}

\author{jiun-chau wang and enzo wendler}

\begin{abstract}
Using martingale convergence theorem, we prove a law of large numbers
for monotone convolutions $\mu_{1}\triangleright\mu_{2}\triangleright\cdots\triangleright\mu_{n}$,
where $\mu_{j}$'s are probability laws on $\mathbb{R}$ with finite
variances but not required to be identical.
\end{abstract}



\date{March 28, 2013}

\address{Department of Mathematics and Statistics, University of Saskatchewan,
106 Wiggins Road, Saskatoon, Saskatchewan S7N 5E6, Canada}

\email{jcwang@math.usask.ca; epw943@mail.usask.ca}

\maketitle

\section{Introduction and the main result}

The \emph{monotone convolution} $\triangleright$ is an associative
binary operation on $\mathcal{M}$, the set of all Borel probability
measures on the real line $\mathbb{R}$. It was introduced by Muraki
in \cite{Muraki1}, based on his notion of monotonic independence
for operators acting on a certain type of Fock space. Later, a universal
construction for the monotone convolution of measures was found in
\cite{Franz}, which does not depend on the underlying Hilbert space.
Thus, the monotone convolution \emph{}$\mu\triangleright\nu$ for
two measures $\mu,\nu\in\mathcal{M}$ is defined as the distribution
of $X+Y$, where the (non-commutative) random variables $X$ and $Y$
are monotonically independent and having distributions $\mu$ and
$\nu,$ respectively. Together with classical, free, and Boolean convolutions,
the monotone convolution is one of the four natural convolution operations
on the set $\mathcal{M}$ \cite{Speicher,Muraki2}.

The research of limit theorems for monotone convolution has been active
in recent years. Notably, an equivalence between monotone and Boolean
limit theorems has been proved in \cite{AnshelevichWilliams}, making it possible to apply the classical Gnedenko type convergence criterion to the weak convergence for sums of monotonically independent
and identically distributed random variables. In spite of these successful
results, the literature lacks a treatment of limit theorems for non-identically
distributed variables. The goal of the current paper is to supply
one such limit theorem in the context of law of large numbers for
variables with finite variances. The convergence condition we discovered
here coincides with the one for the classical law of large numbers.

To explain our result in detail, we first recall some definitions.
Following \cite{GK}, a sequence of measures $\{\nu_{n}\}_{n=1}^{\infty}$
in $\mathcal{M}$ is said to be \emph{stable} if one can find constants
$a_{n}\in\mathbb{R}$ such that \[
\lim_{n\rightarrow\infty}\nu_{n}(\{ t\in\mathbb{R}:\,\left|t-a_{n}\right|\geq\varepsilon\})=0,\qquad\varepsilon>0.\]
Thus, the weak law of Khintchine states that if $\{ X_{n}\}_{n=1}^{\infty}$
is an i.i.d. sequence of real random variables drawn from a law $\mu\in\mathcal{M}$
with finite expectation $m(\mu)$ then the sequence of distributions
\[
D_{1/n}(\underbrace{\mu*\mu*\cdots*\mu}_{n\:\text{times}})\]
is stable, with the asymptotic constants $a_{n}=m(\mu)$ for all $n\geq1$.
Here the notation $*$ means the classical convolution of measures.
For $b>0$, the measure $D_{b}\mu$ is the \emph{dilation} of $\mu$
defined by $D_{b}\mu(A)=\mu(b^{-1}A)$ for Borel measurable $A\subset\mathbb{R}$.

Our main result is the following

\begin{thm}
Let $\{\mu_{n}\}_{n=1}^{\infty}$ be a sequence of probability laws
with finite variances $\text{\emph{var}}(\mu_{n})$. If the series
\begin{equation}
\sum_{k=1}^{\infty}\frac{\text{\emph{var}}(\mu_{k})}{b_{k}^{2}}<\infty\label{eq:1.1}\end{equation}
for some sequence $\{ b_{n}\}_{n=1}^{\infty}$ with $0<b_{1}<b_{2}<\cdots\rightarrow\infty$,
then the sequence of measures \[
D_{1/b_{n}}(\mu_{1}\triangleright\mu_{2}\triangleright\cdots\triangleright\mu_{n})\]
is stable.
\end{thm}
As it will be seen from the proof of Theorem 1.1, a formula for the
asymptotic constants $a_{n}$ here is given by \[
a_{n}=\frac{1}{b_{n}}\sum_{k=1}^{n}m(\mu_{k}).\]
In particular, when $\mu_{1}=\mu_{2}=\cdots=\mu_{n}=\mu$ and $b_{n}=n$,
Theorem 1.1 shows that laws of the normalized sums\[
\frac{Y_{1}+Y_{2}+\cdots+Y_{n}}{n}\]
converge weakly to the point mass at $m(\mu)$ as $n\rightarrow\infty$,
where $\{ Y_{n}\}_{n=1}^{\infty}$ is a monotonically independent
sequence of random variables having the same distribution $\mu$.
This is the weak law of large numbers obtained in \cite{JC}. Finally,
we remark that the condition \eqref{eq:1.1} also implies that the
classical convolutions \[
D_{1/b_{n}}(\mu_{1}*\mu_{2}*\cdots*\mu_{n})\]
are stable (see \cite{GK}).

We end this section with some comments on the method of our proof.
The difficulty in proving limit theorems for monotone convolution
comes from the fact that the computation of $\triangleright$ requires
the composition of certain integral transforms. Precisely, recall
that the \emph{Cauchy transform} of a measure $\mu\in\mathcal{M}$
is defined as \[
G_{\mu}(z)=\int_{-\infty}^{\infty}\frac{1}{z-t}\,\mu(dt),\qquad\Im z>0,\]
and hence the map $F_{\mu}(z)=1/G_{\mu}(z)$ is an analytic self-map
of the complex upper half-plane $\mathbb{C}^{+}=\{ z=x+iy:\, y>0\}$.
For any $\mu,\nu\in\mathcal{M}$, it was shown in \cite{Franz} that
\[
F_{\mu\triangleright\nu}(z)=F_{\mu}\circ F_{\nu}(z),\qquad z\in\mathbb{C}^{+}.\]
Hence, proving Theorem 1.1 amounts to understanding the dynamics of
the backward compositions \[
F_{\mu_{1}}\circ F_{\mu_{2}}\circ\cdots\circ F_{\mu_{n}}\]
of analytic functions. In general, a dynamical system of this sort
is quite complicated to analyze using complex analysis. To go around
this difficulty, we utilize the Markov chain approach in \cite{LetacMalouche}
and the $L^{2}$-martingale convergence theorem to treat the composition
sequence of these $F$-transforms. We shall now begin to present these
details.

\section{Proof of the main result }

Let $\{\mu_{n}\}_{n=1}^{\infty}$ be the given sequence of probability
measures with finite variances, and let $\mathcal{B}$ denote the
Borel $\sigma$-field on $\mathbb{R}$. To each $n\geq2$, we introduce
the function\[
p_{n}(x,B)=\delta_{x}\triangleright\mu_{n}(B),\qquad x\in\mathbb{R},\quad B\in\mathcal{B}.\]
Then each $p_{n}(x,dy)$ is a transition probability function on $\mathbb{R}\times\mathcal{B}$.
Indeed, denote by $H$ the set of bounded and $\mathcal{B}$-measurable
$f:\mathbb{R}\rightarrow\mathbb{R}$ such that the function $Tf$
is also $\mathcal{B}$-measurable, where \[
Tf(x)=\int_{-\infty}^{\infty}f(y)\, p_{n}(x,dy),\qquad x\in\mathbb{R}.\]
First, since the map $x\mapsto\delta_{x}\triangleright\mu_{n}$ is
weakly continuous, the set $H$ contains all continuous and bounded
real-valued functions on $\mathbb{R}$. Secondly, by the monotone
convergence theorem, if $\{ f_{n}\}_{n=1}^{\infty}$ is a monotonically
increasing sequence of nonnegative functions in $H$ which converges
pointwisely to a bounded function $f$, then the limit function $f$
is also in $H$.

Now, consider the set $\mathcal{P}=\{(a,b):\,-\infty\leq a<b\leq\infty\}\cup\{\phi\}$
and the set $\mathcal{L}=\{ B\in\mathcal{B}:\, I_{B}\in H\}$. The
set $\mathcal{P}$ is clearly a $\pi$-system that generates the field
$\mathcal{B}$, and both $\phi,\mathbb{R}$ belong to the set $\mathcal{L}$.
Observe that for any finite interval $(a,b)$, there exists continuous
functions $0\leq f_{n}\leq1$ such that $f_{n}\nearrow I_{(a,b)}$
. This implies that the indicator $I_{(a,b)}$ is in $H$ because
$H$ is closed under bounded monotone convergence. Thus, the set $\mathcal{L}$
contains $\mathcal{P}$. Moreover, it is easy to see that $\mathcal{L}$
is a $\lambda$-system, and therefore we have $\mathcal{L}=\mathcal{B}$
by Dynkin's $\pi$-$\lambda$ theorem. It follows that to each fixed
$B\in\mathcal{B}$, the map $x\mapsto p_{n}(x,B)$ is $\mathcal{B}$-measurable,
justifying that $p_{n}$ is a transition probability.

Next, we consider the real-valued Markov chain $\{ X_{n}\}_{n=1}^{\infty}$
generated by the transition probabilities $\{ p_{n}\}_{n=2}^{\infty}$
and the initial distribution $\mu_{1}$. The existence such a Markov
chain is guaranteed by the Kolmogorov Extension Theorem, and the finite-dimensional
distributions of $\{ X_{n}\}_{n=1}^{\infty}$ are determined by\[
\Pr(X_{j}\in B_{j};\,1\leq j\leq n)=\int_{x_{1}\in B_{1}}\mu_{1}(dx_{1})\int_{x_{2}\in B_{2}}p_{2}(x_{1},dx_{2})\cdots\int_{x_{n}\in B_{n}}p_{n}(x_{n-1},dx_{n}).\]

Notice that one has \begin{eqnarray*}
G_{\mu\triangleright\nu}(z)=G_{\mu}(F_{\nu}(z)) & = & \int_{x\in\mathbb{R}}\frac{1}{F_{\nu}(z)-x}\,\mu(dx)\\
 & = & \int_{x\in\mathbb{R}}G_{\delta_{x}\triangleright\nu}(z)\,\mu(dx)\\
 & = & \int_{x\in\mathbb{R}}\int_{t\in\mathbb{R}}\frac{1}{z-t}\,\delta_{x}\triangleright\nu(dt)\,\mu(dx)\\
 & = & \int_{t\in\mathbb{R}}\frac{1}{z-t}\,\int_{x\in\mathbb{R}}\,\mu(dx)\,\delta_{x}\triangleright\nu(dt)\end{eqnarray*} for any $\mu,\nu\in\mathcal{M}$. (The $\mathcal{B}$-measurability of the function $\delta_{x}\triangleright\nu$
in $x$ follows from the first two paragraphs of this section.) Since
the Cauchy transform $G_{\mu\triangleright\nu}$ determines the measure
$\mu\triangleright\nu$ uniquely, we deduce that \[
\int_{-\infty}^{\infty}\,\mu(dx)\,\delta_{x}\triangleright\nu(dt)=\mu\triangleright\nu(dt).\]
In particular, if $B_{1}=B_{2}=\cdots=B_{n-1}=\mathbb{R}$, an easy
induction argument shows that \begin{eqnarray*}
\Pr(X_{n}\in B_{n}) & = & \int_{-\infty}^{\infty}\,\mu_{1}\triangleright\mu_{2}\triangleright\cdots\triangleright\mu_{n-1}(dx_{n-1})\, p_{n}(x_{n-1},B_{n})\\
 & = & \int_{-\infty}^{\infty}\,\mu_{1}\triangleright\mu_{2}\triangleright\cdots\triangleright\mu_{n-1}(dx_{n-1})\,\delta_{x_{n-1}}\triangleright\mu_{n}(B_{n})\\
 & = & \mu_{1}\triangleright\mu_{2}\triangleright\cdots\triangleright\mu_{n}(B_{n}),\end{eqnarray*}
and therefore the distribution of $X_{n}$ is precisely the monotone
convolution \[
\mu_{1}\triangleright\mu_{2}\triangleright\cdots\triangleright\mu_{n}.\]

We now compute the first two conditional moments of the Markov chain
$\{ X_{n}\}_{n=1}^{\infty}$. The notation $m_{2}(\mu_{n})$ denotes
the second moment of the measure $\mu_{n}$.

\begin{lem}
For $n\geq2$, we have \[
E[X_{n}|\, X_{n-1}]\overset{\text{a.s.}}{=}X_{n-1}+m(\mu_{n})\]
and \[
E[X_{n}^{2}|\, X_{n-1}]\overset{\text{a.s.}}{=}X_{n-1}^{2}+2m(\mu_{n})X_{n-1}+m_{2}(\mu_{n}).\]

\end{lem}
\begin{proof}
We write the function $F_{\mu_{n}}$ in Nevanlinna form:\[
F_{\mu_{n}}(z)=z-m(\mu_{n})+\int_{-\infty}^{\infty}\frac{1}{t-z}\,\sigma_{n}(dt),\]
where $\sigma_{n}$ is a finite Borel measure on $\mathbb{R}$ with
$\sigma_{n}(\mathbb{R})=\text{var}(\mu_{n})$. Because \[
F_{\delta_{x}\triangleright\mu_{n}}(z)=F_{\mu_{n}}(z)-x,\]
the uniqueness of the Nevanlinna representation implies that \[
m(\delta_{x}\triangleright\mu_{n})=m(\mu_{n})+x\]
and $\text{var}(\delta_{x}\triangleright\mu_{n})=\text{var}(\mu_{n})$.
In other words, we have \[
\int_{-\infty}^{\infty}y\, p_{n}(x,dy)=m(\mu_{n})+x\]
and \[
\int_{-\infty}^{\infty}y^{2}\, p_{n}(x,dy)=x^{2}+2m(\mu_{n})x+m_{2}(\mu_{n}).\]
Hence, the desired result follows from the fact that the function
$p_{n}(X_{n-1},dy)$ serves as a regular conditional distribution
for $X_{n}$ given the $\sigma$-subfield $\sigma(X_{n-1})$.
\end{proof}
We are now ready to prove the main result.

\begin{proof}[Proof of Theorem 1.1]
For $n\geq1$, define \[ Y_{n}=X_{n}-\sum_{k=1}^{n}m(\mu_{k}).\] Then
Lemma 2.1 and the Markov property of $\{ X_{n}\}_{n=1}^{\infty}$
imply that $\{ Y_{n}\}_{n=1}^{\infty}$ is a $L^{2}$-martingale.
Consider the martingale differences \[ Z_{n}=Y_{n}-Y_{n-1},\qquad
n\geq2,\] and set $Z_{1}=Y_{1}$. By Lemma 2.1 again, we have
$E[Z_{n}|\, X_{n-1}]=0$ and the second moment \[
E[Z_{n}^{2}]=\text{var}(\mu_{n}).\]

Our proof now follows a classical line. Recall that $\{
b_{n}\}_{n=1}^{\infty}$ is a positive sequence increasing to
$\infty$ for which the condition \eqref{eq:1.1} holds. For $n\geq1$,
let \[ S_{n}=\sum_{k=1}^{n}\frac{Z_{k}}{b_{k}}.\]
 Note that $\{ S_{n}\}_{n=1}^{\infty}$ also forms another $L^{2}$-martingale.
Moreover, we have the second moment\begin{eqnarray*}
E[S_{n}^{\,2}] & = & E[E[(b_{n}^{-1}Z_{n}+S_{n-1})^{2}|\, X_{1},X_{2},\cdots,X_{n-1}]]\\
 & = & b_{n}^{-2}E[E[Z_{n}^{2}|\, X_{n-1}]]+E[S_{n-1}^{\,2}]+2b_{n}^{-1}S_{n-1}E[Z_{n}|\, X_{n-1}]\\
 & = & b_{n}^{-2}\text{var}(\mu_{n})+E[S_{n-1}^{\,2}].\end{eqnarray*}
Proceeding inductively, we get \[
E[S_{n}^{\,2}]=\sum_{k=1}^{n}\frac{\text{var}(\mu_{k})}{b_{k}^{2}},\]
which is bounded uniformly in $n$ by the condition \eqref{eq:1.1}.
Therefore, the $L^{2}$-martingale convergence theorem shows that
$\{ S_{n}\}_{n=1}^{\infty}$ converges almost surely, and Kronecker's
Lemma further implies that \[
\frac{1}{b_{n}}\sum_{k=1}^{n}Z_{k}=\frac{1}{b_{n}}\left[X_{n}-\sum_{k=1}^{n}m(\mu_{k})\right]\rightarrow0\]
on the set of points in the sample space where $\{ S_{n}\}_{n=1}^{\infty}$
converges. Thus, denoting \[
a_{n}=\frac{1}{b_{n}}\sum_{k=1}^{n}m(\mu_{k}),\]
the sequence $b_{n}^{-1}X_{n}-a_{n}$ converges in probability to
$0$ as $n\rightarrow\infty$. Then the proof is completed, because
$b_{n}^{-1}X_{n}$ has distribution $D_{1/b_{n}}(\mu_{1}\triangleright\mu_{2}\triangleright\cdots\triangleright\mu_{n})$.
\end{proof}
\begin{acknowledgement*}
The first author is supported by the NSERC Canada Discovery Grants
and the second author is supported by a University of Saskatchewan
New Faculty Graduate Support Program.
\end{acknowledgement*}


\begin{thebibliography}{1}
\bibitem[1]{AnshelevichWilliams}M. Anshelevich and J. D. Williams,
\emph{Limit theorems for monotonic convolution and the Chernoff product
formula}, Int. Math. Res. Not. IMRN (2013). doi:10.1093/imrn/rnt018

\bibitem[2]{Franz}U. Franz, \emph{Monotone and Boolean convolutions
for non-compactly supported probability measures}, Indiana Univ. Math.
J. \textbf{58} (2009), no. 3, 1151-1185.

\bibitem[3]{GK}B. V. Gnedenko and A. N. Kolmogorov, \emph{Limit distributions
for sums of independent random variables}, Addison-Wesley Publishing
Company, Cambridge, 1954.

\bibitem[4]{LetacMalouche}G. Letac and D. Malouche, \emph{The Markov
chain associated to a Pick function}, Probab. Theory and Relat. Fields
\textbf{118} (2000), 439-454.

\bibitem[5]{Muraki1}N. Muraki, \emph{Monotonic convolution and monotonic
L\'{e}vy-Hin\v{c}in formula}, preprint, 2000.

\bibitem[6]{Muraki2}---------, \emph{The five independences as natural
products}, Infin. Dimens. Anal. Quantum Probab. Relat. Top. \textbf{6}
(2003), no. 3, 337-371.

\bibitem[7]{Speicher}R. Speicher, \emph{On universal products}, Fields
Institute Communications, Vol. \textbf{12} (D. V. Voiculescu, editor),
Amer. Math. Soc., 1997, 257-266.

\bibitem[8]{JC}J.-C. Wang, \emph{Strict limit types for monotone
convolution}, J. Funct. Anal. \textbf{262} (2012), no. 1, 35-58.
\end{thebibliography}
\end{document}